\documentclass[11pt,reqno,oneside]{amsart}
\usepackage{verbatim,amsmath,amssymb,cite,epsfig,xspace,color,subfigure,bbm}
\usepackage[foot]{amsaddr}

\usepackage[margin=1in]{geometry}

\usepackage{tikz}
\usetikzlibrary{decorations.pathreplacing, decorations.pathmorphing, decorations.shapes}

\usepackage{indentfirst}

\numberwithin{equation}{section}

\theoremstyle{plain}
\newtheorem{theorem}{Theorem}[section]
\newtheorem{lemma}{Lemma}[section]
\newtheorem{proposition}{Proposition}[section]
\newtheorem{corollary}{Corollary}[section]
\theoremstyle{definition}

\theoremstyle{remark}
\newtheorem{remark}{Remark}[section]

\newcommand{\calL}{\mathcal{L}}

\newcommand{\abs}[1]{\left| #1 \right|}

\newcommand{\qnlbilinear}{b^{\rm{qnl}}_\del }

\def\R{\mathbb{R}}

\def\cS{\mathcal{S}}

\newcommand{\beq}{\begin{equation}}
\newcommand{\eeq}{\end{equation}}

\newcommand{\ep}{{\epsilon}}
\newcommand{\ga}{{\gamma}}
\newcommand{\del}{\delta}

\newcommand{\Om}{\Omega}
\newcommand{\om}{\omega}

\newcommand{\bm}[1]{\mathbf{#1}}

\definecolor{docol}{rgb}{0.4, 0, 0}

\title[quasinonlocal coupling]{A quasinonlocal coupling method for nonlocal and local diffusion models}

\thanks{{The research of Q. Du and X. Tian is supported in part by the
  U.S.~NSF grants DMS-1719699, AFOSR grant FA9550-14-1-0073 MURI
  Center for Material Failure Prediction through peridynamics and the
  ARO MURI Grant W911NF-15-1-0562.  The work of X. Li is supported in
  part by the Simons Collaboration Grant with Award ID: 426935 and NSF DMS-1720245. The work of J. Lu is supported in part by the National
  Science Foundation under award DMS-1454939.}}

\author{Qiang Du}
\address{Department of Applied Physics and Applied Mathematics,
Columbia University, New York, NY 10027. \email{qd2125@columbia.edu}.}

\author{Xingjie Helen Li}
\address{Department of Mathematics and Statistics, University of North
  Carolina at Charlotte, Charlotte NC 28223.  \email{xli47@uncc.edu}.}
  
\author{Jianfeng Lu}
\address{Department of Mathematics, Department of
  Physics, Department of Chemistry, Duke University, Box 90320, Durham, NC 27708. \email{jianfeng@math.duke.edu}.}

\author{Xiaochuan Tian}
\address{Department of Applied Physics and Applied Mathematics,
Columbia University, New York, NY 10027. \email{xt2156@columbia.edu}. Current address: Department of Mathematics, University of Texas at Austin, Austin, TX 78712.
 \email{xtian@math.utexas.edu}.}

\begin{document}

\begin{abstract}
  In this paper, we extend the idea of ``geometric reconstruction''
  to couple a nonlocal diffusion model directly with the classical local
  diffusion in one dimensional space. This new coupling framework
  removes interfacial inconsistency, ensures the flux balance, and satisfies
   energy conservation as well as
  the maximum principle, whereas none of existing coupling methods
 for nonlocal-to-local
  {coupling} satisfies all of these properties.  We establish the
  well-posedness and provide the stability analysis of the coupling method.
 We investigate the difference to the local limiting problem in terms of the {nonlocal interaction range}.
   Furthermore, we propose a
  first order finite difference numerical discretization and
  perform several numerical tests to confirm the theoretical findings. In
  particular, we show that  the resulting numerical result is free of artifacts near
  the boundary of the domain {where} a classical local boundary condition
  is used, together with a coupled fully nonlocal model in the interior of the
  domain.
\end{abstract}


\maketitle

\section{Introduction}\label{sec_intro}

Nonlocal continuum models
have found interesting applications
in a number of important scientific and engineering problems, for example, the phase transition
\cite{Bates1999,Fife2003}, the nonlocal heat conduction
\cite{Bobaru2010b}, fracture and damage in brittle solids
\cite{Silling2000}. Meanwhile, they can often be linked to
 classic local continuum models where the latter are known to hold \cite{Planas2002,Silling2005,Chasseigne2006a,Silling2008,Bobaru2010a,Lehoucq2010,zhou2010a,Bobaru2011a,Mengesha2014a,Du2012a,Du2013a,Kriventsov,NewsDu,Lipton2014a,Lipton2016a}.

While nonlocal integral-type formulations in a nonlocal continuum model can often
provide a more accurate description of physical systems, especially near defects and singularities, the nonlocality
also increases the computational cost, compared to classical local models
based on partial differential equations (PDEs).
As a result,
it is imperative to employ multiscale methods which can retain accuracy around defect cores while improving efficiency away from
singularities through local continuum descriptions.
In addition, the nonlocal models usually bring
modeling challenges near the boundary, as volumetric boundary
conditions are needed that require additional calibrations
with the physical system.
Improper boundary conditions may create unintended modeling error \cite{Tart2015,DuTao2016,Tian2016b}.
It is thus interesting to explore alternatives that enable the use of
the usual local boundary conditions.

In the past ten years, a number of strategies have been proposed to
couple together local-to-nonlocal or two nonlocal continuum models with
different nonlocality. These coupling methods include (1) Arlequin type domain
decomposition (see e.g.,
\cite{prudhomme:modelingerrorArlequin,Han2012}); (2) Optimal-control
based coupling (see e.g., \cite{Delia2015a}); (3) Morphing
approach (see e.g., \cite{Lub2012a}); (4) Force-based blending mechanism
(see e.g., \cite{Seleson2013a,Seleson2015a}); and (5) Energy-based blending mechanism (see e.g., \cite{Bobaru2008,Silling2015a,Tian2016a}); just to name a few.
Among these multiscale models, some exhibit spurious interfacial forces (``ghost forces'') under uniform strain,
while others forgo the need for energy and develop consistent force-based methods which are non-conservative.

Recently, a new symmetric, consistent and stable coupling strategy for
nonlocal diffusion problems was developed in \cite{LiLu2016} that
couples two nonlocal operators with different horizon parameters
$\delta_1$ and $\delta_2$. The crucial step in the formulation is the
idea of ``geometric reconstruction'' from the
quasinonlocal atomistic-to-continuum method for crystalline solids
(see e.g.,
\cite{Shimokawa:2004,E:2006,MingYang,Shapeev2012a,luskin2011,OrtnerZhang}).
In this paper, we extend the ``geometric reconstruction'' idea to
couple the nonlocal diffusion directly with the classical local
diffusion in one dimensional space. This new framework leads a coupled model
that enjoys linear consistency and preserves the maximum
principle. Furthermore, well-posedness of the coupling problem,
stability analysis and  error estimates are established in this work
to ensure the validity and reliability of the modeling approach and computational results.

Let  us first review {\itshape nonlocal diffusion} equations associated with a positive number $\del$ that characterizes the finite range of nonlocal interaction. We refer to \cite{Du2012a} for more detailed studies on nonlocal diffusion equations. Generically, the spatial
interactions in a linear nonlocal diffusion equation
are characterized by a linear operator $\calL_\del$ acting on a function $u=u(\mathbf{x}):\mathbb{R}^d \rightarrow \mathbb{R}$ such that
\beq
\calL_\del u(\bm x)=2\int_{\R^d}(u(\bm y)-u(\bm x))\ga_\del(\bm x, \bm y) d\bm y, \quad \forall \bm x\in\Om\;,
\eeq
for some open domain $\Om\subset\R^d$. The kernel $\ga_\del$ is usually nonnegative, symmetric and {translational} invariant for isotropic systems.
Often it is {chosen} as a radial function with a compact support,  i.e., $\ga_\del(\bm x, \bm y) =\ga_\del(|\bm x-\bm y|)$ and
$\text{supp}(\ga_\del)\subset B_\del({\bf 0})$, where $B_\del({\bf 0})$ is the $d$-dimensional ball of radius $\del$.
The constant $\del>0$ is often called a horizon parameter that characterizes the range of nonlocality.
We note that the operator $\calL_\del$ can be written in the form of $\calL_\del=\mathcal{D}\gamma_\del \mathcal{D}^*$ where $\mathcal{D}$ and $\mathcal{D}^*$ are some basic nonlocal  operators defined in a nonlocal vector calculus given in \cite{Du2013a}. Such a formulation naturally draws an analogy between the nonlocal operator $\calL_\del$ and the local second order elliptic differential operator {$\nabla\cdot(\bf{C} \nabla)$}. Thus the nonlocal diffusion problems
can be studied and compared with the classical diffusion problems. The nonlocal equations defined on the domain $\Om$
are complemented by the ``Dirichlet type'' boundary conditions, which are constraints on a domain with nonzero $d$-dimensional volume. Thus we arrive at
 the steady-state nonlocal volume-constrained diffusion problem:
\beq \label{eq:ND_steady}
\begin{cases}
-\calL_\del u=f \quad & \text{on } \Om, \\
u= 0 \quad & \text{on }   \Om_{\mathcal{I}}\
\end{cases}
\eeq
for a function $u(\bm x):\R^d\to\R$ and $\Om_{\mathcal{I}}$ being the nonlocal interaction domain of nonzero $d$-dimensional volume.

To make connections of equation \eqref{eq:ND_steady} with their local
differential counterparts, we usually consider the kernel $\ga_\del$
to be suitably  localized as $\del\to0$. Without being too technical, this essentially means that we
want $\ga_\del(|\bm x|)|\bm x|^2$ to be approximating the Dirac delta measure at the origin
as $\del\to0$. Often, a convenient assumption for us to make is that
$\ga_\del$ is a rescaled kernel,
\begin{equation*}
\left \{
\begin{aligned}
& \ga_\del(|\bm x|)=\frac{1}{\del^{d+2}} \gamma\left(\frac{|\bm x|}{\del}\right), \quad \gamma \text{ is nonnegative and nonincreasing on (0,1)},\\
& \text{with } \text{supp}(\ga)\subset [0,1]  \text{ and }  \int_{\R^d} |\bm x|^2 \ga(|\bm x|) d\bm x =d\,.
\end{aligned}
\right.
\text{(K)}
\end{equation*}

In this paper, we propose an energy-based coupling method that combines the nonlocal diffusion equation defined as above with the local classical diffusion equation. Since the construction of our coupling follows the spirit of the quasinonlocal atomistic-to-continuum coupling methods for crystalline materials (see for example, \cite{Shimokawa:2004,E:2006,MingYang,Shapeev2012a,luskin2011,OrtnerZhang}), we call our method the {\itshape quasinonlocal (QNL) coupling} of nonlocal and local diffusion.  We focus on one-dimensional problems in this work to better illustrate the idea. The multi-dimensional generalizations are possible and will be carried out in separate works.

More specifically, in section \ref{sec_coupling} we first define the combined total energy from which the quasinonlocal operator is derived through energy variation, followed by the discussion of the concerned issue of patch-test consistency. Section \ref{sec_stab_wellposedness} contains rigorous arguments of the well-posedness of the coupled problem.
Section \ref{sec_model_error} further explores the modeling accuracy of the coupled method compared with the fully local diffusion equation in terms of small $\del$, in which the uniform first order accuracy in terms of $\del$ is shown. Section \ref{sec_num}
contains numerical experiments and then conclusion and discussions are put in section \ref{sec_conclusion}.


\section{Consistent coupling of nonlocal and local diffusions}\label{sec_coupling}
In this section, we formulate our idea of the QNL coupling in a
one-dimensional bar.  Without loss of generality, we work on the
domain $\Om=(-1,1)$ throughout the paper.  We consider the nonlocal
interaction region to be on the left side of the bar $\Om$ and the
local interaction region to be on the right side with a transition
layer in the middle of width $\delta$.  Now that the domain $\Om$ is
composed of both nonlocal and local interaction regions, the Dirichlet
boundary condition to impose should be considered as a mixture of
nonlocal and local boundary conditions. Specifically, to the left of
the bar $\Om$ there is a nonlocal boundary $(-1-\delta, -1)$ and to
the right of the bar a local boundary $\{ 1\}$. In all further
discussions we use $\Om_\del=(-1-\delta, -1)\cup\{1\}$ as the boundary
domain which is mixed with nonlocal and local boundary. See Figure
\ref{fig:1dbar} for the graphical illustration of the coupled nonlocal
and local domain.

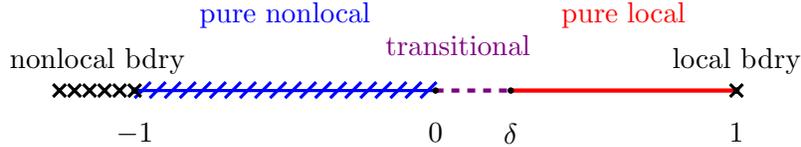
\begin{figure}[htbp]
 \begin{tikzpicture}
 \draw[very thick,blue] (-4,0)--(0,0);
 \draw [decorate, decoration={border,amplitude=0.3cm,segment length=2mm}, blue, very thick] (-4,-0.1)--(0,-0.1);
 \draw[ultra thick, red] (1,0 )--(4,0);
 \draw[ultra thick, dashed, violet] (0,0)--(1,0);
  \draw [decorate, decoration={crosses,shape height=0.17cm, shape width=0.17cm, segment length=2mm}, very thick] (-5,0)--(-4,0);
   \draw [decorate, decoration={crosses,shape height=0.17cm, shape width=0.17cm, segment length=2mm}, very thick] (4,0);

 \node[blue] at (-2, 1) {pure nonlocal};
 \node[red] at (2.5, 1) {pure local};
 \node[violet] at (0.3,0.6) {transitional};
  \node at (-4.5, 0.4) {nonlocal bdry};
  \node at (4, 0.4) {local bdry};

 \draw (-4,0) circle(1pt) [fill] node[below=0.3cm] {$-1$};
 \draw (0,0) circle(1pt) [fill] node[below=0.3cm] {$0$};
 \draw (1,0) circle(1pt) [fill] node[below=0.3cm] {$\delta$};
 \draw (4,0) circle(1pt) [fill] node[below=0.3cm] {$1$};

\end{tikzpicture}
\caption{Graphical illustration of the 1D domain}  \label{fig:1dbar}
\end{figure}

\subsection{The energy space}
The QNL coupling method comes from energy variation of the total energy defined as
\begin{equation}\label{qnl_energy_v1}
E^{\rm{qnl}}_\del(u)
:=\frac{1}{2}\iint_{x\leq0 \text{ or }y\leq0} \gamma_{\delta}(\abs{y-x})\left(u(y)-u(x)\right)^2\,dy dx
+ \frac{1}{2} \int_{x>0}  | u' (x)|^2 \om_\del(x)\,dx.
\end{equation}
where the weight function $\om_\del$ is given by
\begin{equation}\label{local_energy_weight}
\om_\del(x):=\int_{0}^{1} dt \int_{|s|<\frac{x}{t}} |s|^2\gamma_{\delta}(|s|) \, ds.
\end{equation}
From the definition of the kernel $\ga_\del$ in $(K)$, in particular that the second moment of $\ga_\del$ is equal to $d=1$,
it is easy to
see that $\om_\del(x)$ is a nondecreasing function on $[0,\infty)$
with $\om_\del(0)=0$ and $\om_\del(x)=1$ for $x\geq\delta$. Thus the
total quasinonlocal energy has a transition from pure nonlocal to pure
local through the transitional region $(0,\delta)$. We further
characterize of the weight function $\om_\del(x)$ in the following lemma.

\begin{lemma} \label{lem:weight}
By the definition of $\om_\del$ in \eqref{local_energy_weight}, we have the following equations
\begin{align}
\om_\del(x)&= 2 \int_0^x s^2\gamma_\delta(|s|)ds  +2x \int_x^\infty s\gamma_\delta(|s|)ds,  \label{eq:weight}\\
\om_\del^\prime(x) &=2\int_x^{\infty}  s\gamma_{\delta}(s)ds. \label{eq:weight_derivative}
\end{align}
\end{lemma}
\begin{proof}
For the first equation,
\begin{equation*}
\begin{split}
\om_\del(x)&= \int_0^1 dt \int_{|s|<\frac{x}{t}} s^2 \ga_\del(|s|) ds = 2 \int_0^1 dt \int_0^{\frac{x}{t}} s^2 \ga_\del(|s|) ds \\
&=  2\int_0^x s^2 \ga_\del(|s|) \int_0^{1} dt  ds + 2 \int_x^\infty s^2 \ga_\del(|s|) \int_0^{\frac{x}{s}}dt  ds\\
&=2 \int_0^x s^2\gamma_\delta(|s|)ds  +2x \int_x^\infty s\gamma_\delta(|s|)ds\,.
\end{split}
\end{equation*}
Then $\om_\del^\prime(x) $ is obtained by taking derivatives of the expression.
\end{proof}

\begin{remark}
  For given kernel $\gamma$, we could calculate $\om_\del$ using the
  formula \eqref{eq:weight} given in the Lemma \ref{lem:weight}.  We give two
  examples in the following and the plot of the corresponding weight
  function is shown in Figure \ref{fig:weight}. These kernels will be
  used in our numerical example too.

(1) $\gamma_\del(x)=\frac{3}{2\del^3} \chi_{(-\del, \del)}(x)$, then
\begin{equation*}
\om_\del(x)=
\left\{
\begin{aligned}
&\frac{3x}{2\del}-\frac{x^3}{2\del^3}  &&  x\in(0,\delta),\\
& 1  &&  x\geq\delta.
\end{aligned}
\right.
\end{equation*}

(2) $\gamma_\del(x)=\frac{1}{|x|\del^2} \chi_{(-\del, \del)}(x)$, then
\begin{equation*}
\om_\del(x)=
\left\{
\begin{aligned}
&\frac{2x}{\del}-\frac{x^2}{\del^2}  &&  x\in(0,\delta).\\
& 1  &&  x\geq\delta.
\end{aligned}
\right.
\end{equation*}
 \begin{figure}[htbp]
  \centering
 \begin{tikzpicture}[scale=1.5]
      \draw[->, thick] (0,0) -- (3,0) node[below] {$x$};
      \draw[->, thick] (0,0) -- (0,3) node[left] {$\om_\del(x)$};
      \draw[domain=0:1,smooth,variable=\x,blue, very thick] plot ({\x},{3*\x-\x*\x*\x});
       \draw[domain=0:1,smooth,variable=\x,red, very thick, dashed] plot ({\x},{4*\x-2*\x*\x});
      \draw[blue, very thick] (1,2)--(3,2);
      \draw[red, very thick, dashed] (1,2)--(3,2);

      \draw[dashed] (1,2)--(1,0) node[below] {$\delta$};
      \draw[dashed] (1,2)--(0,2) node[left] {$1$};
    \end{tikzpicture}
\caption{Blue line: weight function for $\gamma_\del(x)=\frac{3}{2\del^3} \chi_{(-\del, \del)}(x)$. Red dashed line: weight function for  $\gamma_\del(x)=\frac{1}{|x|\del^2} \chi_{(-\del, \del)}(x)$.} \label{fig:weight}
    \end{figure}
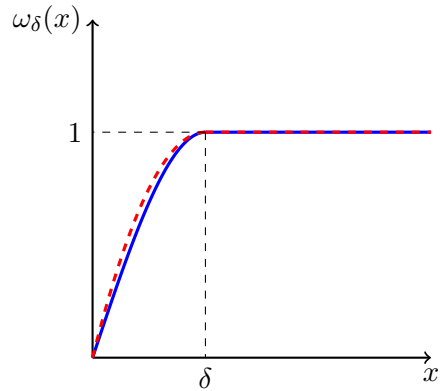
 \end{remark}

The energy defined in \eqref{qnl_energy_v1} has a more intuitive interpretation from
the geometric reconstruction
formulation \cite{E:2006,luskin2011,LiLu2016}. We will show in Proposition~\ref{prop_energy_v2} that \eqref{qnl_energy_v1} is equivalent to the following
\begin{align}\label{qnl_energy_v2}
E^{\rm{qnl}}_\del(u)
=&\frac{1}{2}\iint_{x\leq0 \text{ or }y\leq0} \gamma_{\delta}(\abs{y-x})\left(u(y)-u(x)\right)^2\, dydx \\
&\; + \frac{1}{2}\iint_{x>0\text{ and }y>0}dydx\, \gamma_{\delta}(\abs{y-x} ) \int_{0}^1 dt {\abs{u'(x+t(y-x))}}^2 |y-x|^2.\nonumber
\end{align}

{{To better convey the idea of geometric reconstruction proposed
    in \cite{LiLu2016}, we first assume that
    $\Omega=\Omega_1\sqcup\Omega_2$ is dominated by two different
    nonlocal kernels $\gamma_{\delta_1}$ and $\gamma_{\delta_2}$
    ($\delta_2 < \delta_1$), respectively.  Next, we utilizes the
    interaction kernel $\gamma_{\delta_1}$ throughout the entire
    domain $\Omega$, while in the subregion $\Omega_2$, the
    displacement of bond $(u(y) - u(x))$ will be \emph{reconstructed} so that
    it only involves $x$ and $y$ pairs that are closer in
    distance. More concretely, to link the interaction with kernel
    $\gamma_{\delta_2}$ to $\gamma_{\delta_1}$ where
    $\delta_1 = M \delta_2$, if a bond $\{x-y\}$ is completely
    contained in the subregion $\Omega_2$, then the displacement of
    this bond $(u(y) - u(x))$ will be reconstructed by the following
    expression:
\begin{equation*}
  u(y) - u(x) \rightarrow \left(u \bigl(x + \frac{j+1}{M} (y - x)\bigr) -
    u \bigl(x + \frac{j}{M} (y - x)\bigr) \right) M, \text{ for } j=0,\dots, (M-1).
\end{equation*}
Hence, the bond interaction
 $\gamma_{\delta_2}(\abs{y-x})\left( u(y)-u(x)\right)^2$ in $\Omega_2$ is approximated by
\begin{equation}\label{reconstruct_v1}
 \gamma_{\delta_1}(\abs{y-x}) \frac{1}{M} \sum_{j=0}^{M-1} \left( \left(u\bigl(x+\frac{j+1}{M}(y-x)\bigr)-u\bigl(x+\frac{j}{M}(y-x)\bigr)\right)\frac{\delta_1}{\delta_2}\right)^2.
\end{equation}
Note that if $\abs{x - y} \leq \delta_1$, the difference on the right
is evaluated at points with distance at most
$\frac{\delta_1}{M} = \delta_2$; thus effectively, the difference
$u(y) - u(x)$ is reconstructed by a more local interaction (and hence
the idea was referred to as the ``geometric reconstruction'' scheme in
\cite{E:2006}).  In fact, if such reconstruction is adopted everywhere
in the entire domain $\Omega$, one will recover the fully nonlocal
interactions with kernel $\gamma_{\delta_2}$ only \cite{LiLu2016}.
Notice that when $M=\frac{\delta_1}{\delta_2}\rightarrow \infty$, the
summation in \eqref{reconstruct_v1} can be viewed as a Riemann sum
that converges to an integral, that is
\begin{equation*}
\begin{split}
\frac{1}{M} \sum_{j=0}^{M-1}& \left(\Big(u\bigl(x+\frac{j+1}{M}(y-x)\bigr)-u\bigl(x+\frac{j}{M}(y-x)\bigr)\Big)\frac{\delta_1}{\delta_2}\right)^2\\
=& \sum_{j=0}^{M-1}\left(\frac{u\bigl(x+\frac{j+1}{M}(y-x)\bigr)-u\bigl(x+\frac{j}{M}(y-x)\bigr)}{\frac{1}{M}(y-x)} (y-x)\right)^2 \frac{1}{M}\\
&\qquad \rightarrow \int_{0}^1 {\abs{u'(x+t(y-x))}}^2 |y-x|^2 dt\quad \text{as}\quad M\rightarrow \infty.
\end{split}
\end{equation*}
The nonlocal bond interaction
$\gamma_{\delta}(\abs{y-x})\left( u(y)-u(x)\right)^2$ can be
reconstructed by its local continuum approximation:
\begin{equation}\label{cont_approx}
\gamma_{\delta}(\abs{y-x} )\cdot \int_{0}^1{\abs{u'(x+t(y-x))}}^2 |y-x|^2 dt.
\end{equation}
Based on this construction, we arrive at the total coupling energy
\eqref{qnl_energy_v2}.  }}

We will show now that the two ways of writing the quasinonlocal total energy are the same.
From the expressions \eqref{qnl_energy_v1} and \eqref{qnl_energy_v2}, it suffices to show that local contribution to the total energy is equivalent. The two different ways of writing the local contribution of the energy has their own advantages and we will adopt either definition at our convenience in the sequel.

\begin{proposition}\label{prop_energy_v2}
The following two  expressions of local contribution to the total energy are equivalent
\begin{align}
& \label{prop_energy_eq1}
E^{\rm{loc}}_\del(u)=
\frac{1}{2}\iint_{x>0\text{ and }y>0}dxdy\, \gamma_{\delta}(\abs{y-x} )\cdot \int_{0}^1 dt {\abs{u'(x+t(y-x))}}^2 |y-x|^2, \\
\intertext{and,}
& \label{eq:loc_energy}
E^{\rm{loc}}_\del(u)=\frac{1}{2} \int_{x>0} |u'(x)|^2 \om_\del(x)\,dx.
\end{align}
\end{proposition}
\begin{proof}
We start with recasting the right hand side of \eqref{prop_energy_eq1}
\begin{align}\label{prop_energy_eq2}
 \frac{1}{2}&\iint_{x>0\text{ and }y>0} \gamma_{\delta}(\abs{y-x} )\cdot \int_{0}^1 dt {\abs{u'(x+t(y-x))}}^2 |y-x|^2\nonumber\\
&= \frac{1}{2} \int_{0}^{1} dt \int_{x>0}dx\int_{z> (1-t)x} dz \gamma_{\delta}\left( \abs{\frac{z-x}{t} }\right) |u'(z)|^2\frac{1}{t^3}|z-x|^2\nonumber\\
 &=\frac{1}{2}\int_{0}^{1} dt \int_{z>0}dz |u'(z)|^2 \int_{0<x<\frac{z}{1-t}} \gamma_{\delta}\left( \abs{\frac{x-z}{t} }\right)  \frac{1}{t^3}|x-z|^2~dx\nonumber \\
&=\frac{1}{2} \int_{z>0}dz |u'(z)|^2\int_{0}^{1} dt\int_{-\frac{z}{t}<s<\frac{z}{1-t}} \gamma_{\delta}\left( |s|\right)  |s|^2~ds \nonumber\,.
\end{align}
Now since
\begin{equation*}
\begin{split}
&\int_{0}^{1} dt\int_{-\frac{z}{t}<s<\frac{z}{1-t}} |s|^2\gamma_{\delta}\left( |s|\right)~ds\\
=& \int_{0}^{1} dt\int_{-\frac{z}{t}<s<0} |s|^2\gamma_{\delta}\left( |s|\right)~ds+\int_{0}^{1} dt\int_{0<s<\frac{z}{1-t}}|s|^2 \gamma_{\delta}\left( |s|\right)~ds \\
=&\int_{0}^{1} dt\int_{-\frac{z}{t}<s<0} |s|^2\gamma_{\delta}\left( |s|\right)~ds+\int_{0}^{1} dt\int_{0<s<\frac{z}{t}}|s|^2 \gamma_{\delta}\left( |s|\right)~ds\,,
\end{split}
\end{equation*}
we arrive at  definition of $E^{\rm{loc}}_\del$ in \eqref{eq:loc_energy} with the weight function $\om_\del$ as in \eqref{local_energy_weight}.
\end{proof}

Naturally, we seek solutions in the energy space $\cS^{\rm{qnl}}_\del(\Om)$ equipped with norm
\begin{equation*}
\| u\|^2_{\cS^{\rm{qnl}}_\del(\Om)}= \| u\|^2_{L^2(\Om\cup\Om_\del)}+| u |^2_{\cS^{\rm{qnl}}_\del(\Om)}
\end{equation*}
where  $| u |^2_{\cS^{\rm{qnl}}_\del(\Om)}:= 2E^{\rm{qnl}}_\del(u)$.  Now define $\cS^{\rm{qnl}}_\del(\Om)$ to be the completion of
$C^\infty_c(\Om)$ under the norm $\| \cdot \|_{\cS^{\rm{qnl}}_\del(\Om)}$, namely,
\begin{equation*}
\cS^{\rm{qnl}}_\del(\Om)=\{ u\in L^2(\Om\cup\Om_\del) : \exists \{u_n\}\in C^\infty_c(\Om),  \| u_n-u\|_{\cS^{\rm{qnl}}_\del(\Om)}\to 0 \text{ as } n\to\infty \}\,.
\end{equation*}
 Then we know first that $\cS^{\rm{qnl}}_\del(\Om)$ is a Hilbert space with
 inner product $(\cdot, \cdot)_{\cS^{\rm{qnl}}_\del(\Om)}$ to be defined as
 \begin{equation*}
 (u, v)_{\cS^{\rm{qnl}}_\del(\Om)}=(u, v)_{L^2(\Om\cup\Om_\del)} + b_\del^{\rm{qnl}}(u, v)
 \end{equation*}
where $b_\del^{\rm{qnl}}(u, v)$ is defined as
\beq\label{qnl_bilinear}
\begin{split}
b^{\rm{qnl}}_\del (u, v)=&\iint_{x\leq0 \text{ or }y\leq0}\gamma_{\delta}(\abs{y-x})\left(u(y)-u(x)\right)\left(v(y)-v(x)\right)~dydx\\
&\qquad\qquad + \int_{x>0} u'(x) v'(x)\om_\del(x)\, dx.
\end{split}
\eeq
Moreover, Poincar\'e type inequality holds  on the space $\cS^{\rm{qnl}}_\del(\Om)$ that is crucial in showing the well-posedness of the variational problem.

\begin{proposition}[Poincar\'e inequality] \label{prop:poincare}
For $u\in \cS^{\rm{qnl}}_\del(\Om)$, we have the following Poincar\'e type inequality,
\beq \label{eq:poincare}
\| u\|_{L^2(\Om)}\leq C | u |_{\cS^{\rm{qnl}}_\del(\Om)}\,,
\eeq
where $C$ is independent of $u$.
\end{proposition}
\begin{proof}
From Proposition \ref{prop_stab} which will be shown later in section \ref{sec_stab_wellposedness}, we know that the quasinonlocal energy $| u |_{\cS^{\rm{qnl}}_\del(\Om)}$
is bounded from below by a purely nonlocal energy defined on the entire domain $\Om$. Thus by the nonlocal Poincar\'e inequality established previously in early works, e.g., \cite{Du2012a, Mengesha2014a},
\eqref{eq:poincare} is true.
{Indeed, \cite{Mengesha2014a} shows that for a given small number $\ep$ there exists $\del_0(\ep)$ such that for all $\del<\del_0$ the lemma holds
with $C(\del_0)=A+\ep$, where $A$ is the classical local Poincar\'e constant for the domain $\Om$.}
\end{proof}


\subsection{The QNL operator} \label{subsec:qnl_op}
We will derive the QNL operator denoted as $\calL^{\rm{qnl}}_\del $
from energy variation.
We take the first variation of $E^{\rm{qnl}}_\del(u)$ in \eqref{qnl_energy_v2} with any test function $v\in C_c^\infty(\Om)$, and get
\begin{align}\label{first_var_qnl_eq1}
&\langle d E^{\rm{qnl}}_\del (u), v\rangle := \lim_{\ep\to0} \frac{E^{\rm{qnl}}_\del (u+\ep v)- E^{\rm{qnl}}_\del (u) }{\ep}  \\
&\,= \iint\limits_{x\leq0 \text{ or }y\leq0} \gamma_{\delta}(\abs{y-x})\left(u(y)-u(x)\right)\left(v(y)-v(x)\right) dydx+ \int\limits_{x>0}  \om_\del(x) u'(x) v'(x) dx\nonumber\\
&\,=- 2\iint\limits_{x\leq0 \text{ or }y\leq0} \gamma_{\delta}(\abs{y-x})\left(u(y)-u(x)\right) v(x) dy dx
- \int\limits_{x>0}  (\om_\del(x) u'(x))' v(x) dx,\nonumber
\end{align}
where the last equality comes integration by parts and the fact that $\om_\del(0)=0$.
The force formalism $\calL^{\rm{qnl}}_\del u(x)$ is negative to the first variation of total energy, and it splits into three cases:

\begin{itemize}
\item Case I (nonlocal region): for $x\leq0$,
\end{itemize}
\begin{align}\label{force_case1}
\calL^{\rm{qnl}}_\del u(x)
=& 2 \int_{y\in \mathbb{R}} \gamma_{\delta} (\abs{y-x})\left(u(y)-u(x)\right) dy.
\end{align}

\begin{itemize}
\item Case II (transitional region): for $0<x\leq\delta$,
\end{itemize}
\begin{align}\label{force_case2}
\calL^{\rm{qnl}}_\del u(x)
=&2\int_{y<0}
 \gamma_{\delta}(\abs{y-x})\left(u(y)-u(x)\right)dy+ (\om_\del(x) u'(x))'\,.
  \end{align}

\begin{itemize}
\item Case III (local region): for $x>\delta$, and since $\om_\del(x) = 1$ for $x \geq \delta$,
\end{itemize}
\begin{align}\label{force_case3}
\calL^{\rm{qnl}}_\del u(x)
=& (\om_\del(x) u'(x))' = u''(x)\,.
\end{align}

\begin{remark}\label{rem:BalanceLM}
  { Since the QNL operator $\calL^{\rm{qnl}}_\del $ is defined
    through the first variation of total energy,
    $\calL^{\rm{qnl}}_\del $ is self-adjoint, that is, from a physical
    point of view, the force acting on $x$ from $y$ is equal to the
    force acting on $y$ from $x$. This symmetry in acting forces
    guarantees the balance of linear momentum. In addition, this QNL
    framework ensures the flux balance, and satisfies energy
    conservation.  }
\end{remark}


\subsection{Consistency at the interface}
\label{subsec:ghost_force}

We will show in this part that the QNL coupling is consistent at the interface (in the language of atomistic-to-continuum coupling, it is free of ghost force), namely, for a linear displacement $u^{\rm{lin}}(x)=Fx+a$, the force equals zero. For this matter, we only need to worry about the values of $\calL^{\rm{qnl}}_\del u^{\rm{lin}}$ in the interfacial region, since it is obviously zero in the pure nonlocal and local regions as given by case I and case III in \eqref{force_case1} and \eqref{force_case3}.  For a more general consideration that will also be useful in the next sections, we give the following lemma that involves the operator $\calL^{\rm{qnl}}_\del $ acting on smooth functions in the interfacial region. The lemma states that if $\delta$ is small, the QNL diffusion is approximately a local diffusion with effective diffusion constant $a(x)$.

\begin{lemma} \label{lem:interfacial}
 For any smooth function $v$,
\begin{equation}\label{eq:interfacial}
\calL^{\rm{qnl}}_\del v(x)=a(x)v''(x) + O(\del \| v'''\|_{C^0} ),\quad 0< x < \delta\,,
\end{equation}
where $a$ is given by
\begin{equation} \label{equiv_diff}
a(x)=1- \int_x^\delta s^2 \ga_\del(|s|) ds +2x \int_x^\del s\gamma_\delta(|s|)ds\,.
\end{equation}
\end{lemma}
\begin{proof}
For $x\in(0,\delta)$, by the expressions of $\om_\del$ and $\om_\del^\prime$ in Lemma~\ref{lem:weight}, we have
\begin{equation*}
\begin{split}
\calL^{\rm{qnl}}_\del v(x)&= 2\int_{y<0}
 \gamma_{\delta}(\abs{y-x})\left(v(y)-v(x)\right)dy+ (\om_\del(x) v'(x))' \\
 &=2\int_{-\delta}^{-x} \ga_\del(s)\left( sv'(x)+\frac{1}{2} s^2 v''(x)+O(|s|^3 \| v'''\|_{C^0}) \right)\\
  &\qquad \qquad +  \om_\del(x) v''(x) +\om_\del^\prime(x) v'(x) \\
 &=  \left(\int_x^\delta s^2 \ga_\del(|s|) ds \right) v''(x) +  \om_\del(x) v''(x)  + O(\delta \| v'''\|_{C^0})\\
 &= \left( 1- \int_x^\delta s^2 \ga_\del(|s|) ds +2x \int_x^\del s\gamma_\delta(|s|)ds \right) v''(x) +O(\delta\| v'''\|_{C^0})\,.
 \end{split}
\end{equation*}
Thus, we proved this lemma.
\end{proof}

\begin{remark}\label{rem:equiv_diff}
We can further quantify $a(x)$ as follows.
\begin{enumerate}
\item One can show that $\frac{1}{2}\leq a(x)\leq \frac{3}{2}$ for $x\in (0,\delta)$ and $a(\delta)=1$. Indeed,
\begin{equation*}
a(x)\geq 1- \int_x^\delta s^2 \ga_\del(|s|) ds  \geq 1- \int_0^\delta s^2 \ga_\del(|s|) ds =\frac{1}{2}\,,
\end{equation*}
and
\begin{equation*}
a(x)\leq 1- \int_x^\delta s^2 \ga_\del(|s|) ds +2 \int_x^\del s^2\gamma_\delta(|s|)ds  \leq 1+ \int_0^\delta s^2 \ga_\del(|s|) ds=\frac{3}{2}\,.
\end{equation*}
As last, $a(\delta)=1$ is obvious.
\item For the two examples that $\gamma_\del(x)=\frac{3}{2\del^3} \chi_{(-\del, \del)}(x)$ and $\gamma_\del(x)=\frac{1}{|x|\del^2} \chi_{(-\del, \del)}(x)$, we could calculate $a(x)$ explicitly through equation \eqref{equiv_diff}.
\begin{equation*}
a(x)=
\left\{
\begin{aligned}
& \frac{1}{2}+\frac{3x}{2\delta}-\frac{x^3}{\delta^3} \quad \text{for } \gamma_\del(x)=\frac{3}{2\del^3} \chi_{(-\del, \del)}(x)\\
&\frac{1}{2}+\frac{2x}{\delta}-\frac{3x^2}{2\delta^2} \quad \text{for } \gamma_\del(x)=\frac{1}{|x|\del^2} \chi_{(-\del, \del)}(x) \,.
\end{aligned}
\right.
\end{equation*}

We remark that although the effective local diffusion coefficient $a(x)$ is not equal to a constant one for $0< x<\delta$, we
have in the two cases
\begin{equation*}
\int_0^{\delta}a(x) dx=\delta\,.
\end{equation*}
In other words, the spacial averaged diffusion coefficient for $0< x< \delta$ is equal to one.
\end{enumerate}
\end{remark}

Lemma \ref{lem:interfacial} shows the expansion of $\calL^{\rm{qnl}}_\del  v(x)$ in the interfacial region using with the second and higher derivatives of $v$. Thus it
is obvious that for a linear function $u^{\rm{lin}}$, $\calL^{\rm{qnl}}_\del u^{\rm{lin}}=0$. In other words, the QNL coupling
passes the patch-test.
\begin{corollary}[Patch-test consistency]
For a linear function $u^{\rm{lin}}(x)=Fx+a$,
\begin{equation*}
\calL^{\rm{qnl}}_\del u^{\rm{lin}}=0\,.
\end{equation*}
\end{corollary}
\begin{proof}
This immediately follows from  \eqref{force_case1}, \eqref{force_case2},  and \eqref{force_case3} using Lemma \ref{lem:interfacial}.
\end{proof}

\section{Stability and well-posedness}\label{sec_stab_wellposedness}
In this section, our goal is to show that the bilinear form $b^{\rm{qnl}}_\del (\cdot, \cdot): \cS^{\rm{qnl}}_\del(\Om) \times \cS^{\rm{qnl}}_\del(\Om) \to \mathbb{R}$
defined by \eqref{qnl_bilinear_stab} is bounded and
coercive, thus the well-posedness of the variational problem can be followed.
The boundedness of the bilinear norm is obvious since $\cS^{\rm{qnl}}_\del(\Om)$ is a Hilbert space and $b^{\rm{qnl}}_\del (\cdot, \cdot)$ is part of its inner product. The coercivity is from the Poincar\'e inequality \eqref{eq:poincare}, and the essential step is proved in
Proposition \ref{prop_stab}. Now let us define the local contribution of the bilinear form as
\begin{equation}\label{half_dom_local_bilinear}
b^{\rm{loc}}_\del (u,v):=  \int_{x>0} u'(x) v'(x)\om_\del(x)\, dx.
\end{equation}

We can see the lower bound of $b^{\rm{loc}}_\del (u,u)$ in the following lemma.

\begin{lemma}\label{lem:local_bilinear_stab}
For $b^{\rm{loc}}_\del (u,v)$ defined in \eqref{half_dom_local_bilinear}, we have
\begin{equation}\label{stab_local_bilinear_eq}
b^{\rm{loc}}_\del (u,u)
\ge \iint_{x>0\text{ and }y>0}\gamma_{\delta}(\abs{y-x} )\big( u(y)-u(x) \big)^2~dxdy .
\end{equation}
\end{lemma}
\begin{proof}

The right hand side of \eqref{stab_local_bilinear_eq} can be recast as
\begin{align}\label{stab_local_bilinear_eq3}
&\int_{x>0\text{ and }y>0}\gamma_{\delta}(\abs{y-x} )\big( u(y)-u(x) \big)^2~dxdy\nonumber\\
&\quad= \int_{x>0}dx\int_{y>0}dy \gamma_{\delta}(\abs{y-x} )\left[\int_{0<t<1} du\big(x+t(y-x)\big)\right]^2\nonumber\\
&\quad= \int_{x>0}dx\int_{y>0}dy \gamma_{\delta}(\abs{y-x} )\left[\int_{0}^1 (y-x)\cdot u'\big(x+t(y-x)\big) dt \right]^2\nonumber\\
&\quad \leq   \int_{x>0}dx\int_{y>0}dy \gamma_{\delta}(\abs{y-x} ) (y-x)^2 \int_{0}^1 |u'\big(x+t(y-x)\big)|^2 dt\,,
\end{align}
where the last expression is exactly  $2 E^{\rm{loc}}_\del(u)=b^{\rm{loc}}_\del (u,u)$ as shown in Proposition \ref{prop_energy_v2}.
\end{proof}
Lemma~\ref{lem:local_bilinear_stab} immediately leads to the stability property compared to the fully nonlocal bilinear operator.
\begin{proposition}\label{prop_stab}
For $b^{\rm{qnl}}_\del (u,v)$ defined in \eqref{half_dom_local_bilinear}, we have
  \begin{equation}\label{qnl_bilinear_stab}
    \qnlbilinear(u,u)
    \ge \iint_{x,y\in\mathbb{R}}\gamma_{\delta}(\abs{y-x})\left(u(y)-u(x)\right)^2~dydx.  \end{equation}
\end{proposition}
\begin{proof}
Recall the definition of $b^{\rm{qnl}}_\del (u, u)$ and use the conclusion of Lemma~\ref{lem:local_bilinear_stab}, we immediately get
\begin{align*}
b^{\rm{qnl}}_\del (u, u)=&\int_{x\leq0 \text{ or }y\leq0}\gamma_{\delta}(\abs{y-x})\left(u(y)-u(x)\right)^2~dxdy
+b^{\rm{loc}}_\del (u,u)\\
&\ge \int_{x\leq0 \text{ or }y\leq0}\gamma_{\delta}(\abs{y-x})\left(u(y)-u(x)\right)^2~dxdy\\
&\qquad+\int_{x>0\text{ and }y>0}\gamma_{\delta}(\abs{y-x} )\big( u(y)-u(x) \big)^2~dxdy
\end{align*}
\end{proof}

Now from the Poincar\'e inequality Proposition \ref{prop:poincare}, we conclude that $b^{\rm{qnl}}_\del (\cdot,\cdot)$ is bounded and coercive, thus leading to the well-posedness
of the QNL model.

\begin{theorem}\label{thm_wellposedness}
The QNL diffusion equation given by
\beq \label{eq:QNLD_steady}
\begin{cases}
-\calL^{\rm{qnl}}_\del  u^{\rm{qnl}}_\delta (x)=f(x), & \text{for } x\in \Om\\
u_\delta(x)=0, &\text{for }  x\in \Om_\del
\end{cases}
\eeq is well-posed, where $\calL^{\rm{qnl}}_\del $ is defined in
subsection~\ref{subsec:qnl_op}.
\end{theorem}
\begin{proof}
The well-posedness follows immediately from Lax-Milgram theorem.
\end{proof}

\section{Convergence to the local diffusion as $\delta \to 0$}\label{sec_model_error}

We consider in this section the modeling error estimate of the QNL coupling equation \eqref{eq:QNLD_steady} as $\delta\to0$
to the local differential equation
\beq \label{eq:LD_steady}
\begin{cases}
- u_0^{\prime\prime} (x)=f(x), & x\in \Om\\
u_0(-1)=u_0(1)=0\,. &
\end{cases}
\eeq

{
In this section we assume that
$u_0$ has a smooth zero extension into  $(-1-\delta,-1)$ to avoid discussions on the effect of nonlocal boundary condition there.
We denote the error between the solutions to \eqref{eq:QNLD_steady} and \eqref{eq:LD_steady} to be
 $e_\delta(x)=u^{\rm{qnl}}_\delta -u_0(x)$.
With this extension and both local and nonlocal homogeneous Dirichlet conditions imposed on $u^{\rm{qnl}}_\delta$
on the interval $(-1-\delta, -1)$ and  the right end point $1$ of $\Om$ respectively,
we see that $e_\delta(x)=0$ for $x\in \Om_\del$.}

\paragraph{\textbf{Truncation error}}
Let the truncation error be $T_\delta(x)=\calL^{\rm{qnl}}_\del   u_0(x) - u_0^{\prime\prime}(x)$. Then $T_\delta(x)=T_\delta^1(x)+T_\delta^2(x)$,
where $T_\delta^1(x)=T_\delta(x)\chi_{(-1,0)} (x)$ and $T_\delta^2(x)=T_\delta(x) \chi_{(0,\delta)} (x)$. According to the calculations in section \ref{subsec:ghost_force}, we know that  $T_\delta^1(x)=O(\delta^2)$ for $x\in(-1,0)$ and $T_\delta^2(x)=O(1)$ for $x\in(0,\delta)$.  Notice that from Lemma \ref{lem:interfacial}, for $x\in(0,\delta)$,
\begin{equation*}
\begin{split}
T_\delta^2(x)=\calL^{\rm{qnl}}_\del   u_0(x) - u_0^{\prime\prime}(x) &=a(x) u_0^{\prime\prime}(x)- u_0^{\prime\prime}(x) +O(\delta) \\
&=(a(x)-1) u_0^{\prime\prime}(x)+O(\delta)\,.
\end{split}
\end{equation*}
Since $\frac{1}{2}\leq a(x)\leq \frac{3}{2} $ by Remark \ref{rem:equiv_diff}, we have
\begin{equation} \label{T2estimate}
|T_\delta^2(x)| \leq \frac{1}{2}C^\ast +O(\delta)\,,
\end{equation}
where $C^\ast=\|u_0\|_{C^2}$.
Now that $-\calL^{\rm{qnl}}_\del e_\delta(x)=-\calL^{\rm{qnl}}_\del   u^{\rm{qnl}}_\delta(x) +\calL^{\rm{qnl}}_\del u_0(x) = T_\delta(x)$,
we have $e_\delta(x) = (-\calL^{\rm{qnl}}_\del )^{-1}T_\delta^1(x)+ (-\calL^{\rm{qnl}}_\del )^{-1}T_\delta^2(x) =e_\delta^1(x)+e_\delta^2(x)$, where $e_\delta^1(x)$ and $e_\delta^2(x)$ are defined as
\beq \label{eq:errors}
\begin{cases}
e_\delta^1(x)=(-\calL^{\rm{qnl}}_\del )^{-1}T_\delta^1(x), &\\
e_\delta^2(x)=(-\calL^{\rm{qnl}}_\del )^{-1}T_\delta^2(x)\,.&
\end{cases}
\eeq
We are going to show next that $|e_\delta^1(x)| =O(\delta^2)$ and $|e_\delta^2(x)|=O(\delta)$. Thus the total error is of order $O(\delta)$. The main ingredients
are maximum principle and barrier functions.

In the following, we will show a maximum principle for solutions of \eqref{eq:QNLD_steady} that may have discontinuity at $0$.  We need such result
for error estimate because the truncation error $T_\del$ has been decomposed into two piecewise smooth functions such that  $e_\del^1$ and $e_\del^2$ might be discontinuous at $0$.

\begin{lemma}[Maximum principle]\label{lem:max_prip}
The operator $\calL^{\rm{qnl}}_\del $ satisfies the maximum principle, namely,
if $u\in C([-1-\del,0])\cap C^2([0,1])$, then $-\calL^{\rm{qnl}}_\del  u (x)\leq  0$ in $\Om$ implies that,
\begin{equation*}
\max_{x \in \Omega\cup\Om_\del } u(x) \leq \max_{x\in\Om_\del } u(x).
\end{equation*}

\end{lemma}
\begin{proof}
First, from $-\calL^{\rm{qnl}}_\del  u (x)\leq  0$ in $(0,1)$ we can show that
\begin{equation}\label{eq:maxproof1}
\max_{x \in(0,1)} u(x)\leq \max_{x\in\{ 0^+\}\cup\{1\}} u(x)\,,
\end{equation}
where $u(0^+)=\lim_{x\to 0, x>0} u(x)$. Indeed, if we assume the opposite is true, namely if $\tilde x\in (0,1)$ is an isolated maximum point, then we
must have $u^\prime(\tilde x)=0$ and $u^\prime(\tilde x)<0$. From the expressions of $\calL^{\rm{qnl}}_\del$ in \eqref{force_case2} and \eqref{force_case3}, we have immediately $-\calL^{\rm{qnl}}_\del  u (\tilde x)>0$, which contradicts the assumption.
 Second, from $-\calL^{\rm{qnl}}_\del  u (x)\leq  0$ in $[-1,0]$ we could show
\begin{equation}\label{eq:maxproof2}
\max_{x \in(-1-\delta,\delta)} u(x)\leq \max_{x\in(-1-\delta,-1)\cup (0,\delta)} u(x)\,.
\end{equation}
The argument is the following. Assume the opposition is true, namely,
\begin{equation*}
\max_{x \in(-1-\delta,\delta)} u(x) > \max_{x\in(-1-\delta,-1)\cup (0,\delta)} u(x),
\end{equation*}
then we could find $x^\ast\in[-1,0]$ such that $u(x^\ast)= \max_{x \in(-1,0)} u(x)$ and
 \begin{equation*}-\calL^{\rm{qnl}}_\del  u (x^\ast)= -\int_{-\delta}^{\delta} \gamma_\delta(|s|)(u(x^\ast+s)-u(x^\ast))ds>0,
 \end{equation*}
 which gives us a contradiction. So $u$ has to satisfy \eqref{eq:maxproof2}.

 Now combine the result of \eqref{eq:maxproof1} and \eqref{eq:maxproof2},
we only need to show  $u(0^+)\leq \max_{x \in \Om_\del } u(x)$. Assume the opposite,
 namely $u(0^+)> u(x)$ for any $x\in [-1-\delta, 0^-] \cup(0,1]$. Then since $u(0^+)>u(0^-)$,
we have  $\int_{y<0}\gamma_{\delta}(\abs{y-x})\left(u(y)-u(x)\right)dy<0$ for sufficiently small $x>0$.
Considering also that $u^\prime(0^+)\leq 0$ (since $u(0^+)> u(x)$ for any $x>0$) and $\om_\del(0^+)=0 $,  we see that for small enough $x>0$,
 \begin{equation*}
 -\calL^{\rm{qnl}}_\del  u (x) = -2\int_{y<0}
 \gamma_{\delta}(\abs{y-x})\left(u(y)-u(x)\right)dy- \om_\del(x) u^{\prime\prime}(x) -  \om_\del^\prime(x)  u^{\prime}(x) >0\,,
 \end{equation*}
 which gives us a contradiction.

Hence, we proved the lemma.
\end{proof}

\begin{theorem} Suppose $u^{\rm{qnl}}_\delta$ and $u_0$ are strong solutions to \eqref{eq:QNLD_steady} and \eqref{eq:LD_steady} respectively.
Assume that $u_0\in C^3(\overline{\Om\cup\Om_\del})$, then
\begin{equation*}
\| u^{\rm{qnl}}_\delta(x)-u_0(x)\|_{L^\infty(\Om)} = O(\delta)\,.
\end{equation*}
\end{theorem}
\begin{proof}
We will construct barrier functions of nonnegative values on $\Om\cup\Om_\del$ and then estimate $e_\delta^1(x)$ and $e_\delta^2(x)$ defined by \eqref{eq:errors}.
The first barrier function is a simple quadratic function. Take $\Phi_1(x)=-c x^2+4c$, then from the calculations in section \ref{subsec:ghost_force} we know that $-\calL^{\rm{qnl}}_\del   (\delta \Phi_1(x) )\geq c\delta $. For $u_0\in C^3(\overline{\Om\cup\Om_\del})$, we know that $T_\delta^1(x)$ is at least of order $O(\del)$, so by choosing $c$ large enough we could have $c\delta \geq T_\delta^1(x)$.
Now from Lemma  \ref{lem:max_prip} we conclude that
\begin{equation*}
\max_{x \in \Om\cup\Om_\del } (e_\delta^1(x) - \delta \Phi_1(x))   \leq  \max_{x\in\Om_\del } (e_\delta^1(x) - \delta \Phi_1(x))  \leq 0\,,
\end{equation*}
so we have $e_\delta^1(x) \leq \delta \Phi_1(x) \leq 4c
\delta$.
Applying the same arguments to $-e_\delta^1(x) $ we also have
$-e_\delta^1(x) \leq 4c \delta$. Thus $|e_\delta^1(x)|=O(\delta)$.

The second barrier function $\Phi_2(x)$ is more carefully designed in
order to get the estimate of $e_\delta^2(x)$.  The key is to define
$\Phi_2(x)$ such that $\Phi_2 \in C([-1-\del,0])\cap C^2([0,1])$ (so as to use the maximum principle) and it is linearly decaying to zero outside the
interfacial region.  We define the barrier function
$\Phi_2(x)$ to be
 \begin{equation}\label{Phi2}
\Phi_2(x)=
\left\{
\begin{aligned}
&\delta x+\delta+\delta^2  \quad x\in (-1-\delta, 0)\\
&\frac{1}{8\delta} x^3 - \frac{3}{4}x+\frac{1}{2}\del x +\del +\del^2\quad x\in [0, 2\delta]\\
&-\delta x+\delta+2\delta^2    \quad x\in [2\delta, 1) \,.
\end{aligned}
\right.
\end{equation}
One could check that $\Phi_2\in C([-1-\del,0])\cap C^2([0,1])$ and $-\calL^{\rm{qnl}}_\del   ( \Phi_2(x) )\geq 0$ for $x\in(-1,1)$.
In particular, for $x\in(0,\delta)$, after taking Taylor-expansion, we can write
\begin{equation*}
\begin{split}
-\calL^{\rm{qnl}}_\del   ( \Phi_2(x) )
 =&-a(x)\Phi_2^{\prime\prime}(x)-2\int_{-\delta}^{-x} \ga_\del(s)\left( \frac{1}{6} s^3 \Phi_2^{\prime\prime\prime}(x)  ds\right) \\
 =&-a(x) \left(\frac{3}{4} \frac{x}{\del}-\frac{3}{2}\right)- \frac{1}{3}\cdot  \frac{3}{4\del} \int_{-\delta}^{-x} s^3\gamma_\delta(s) ds\\
\geq &\, \frac{3}{4} a(x) +  \frac{1}{4\del} \int_{\delta}^{x} s^3\gamma_\delta(s) ds\geq  \frac{3}{8}\,,
\end{split}
\end{equation*}
where the last inequality comes from the fact that $a(x)\geq\frac{1}{2}$.
Then by  the expression of $T_\delta^2(x)$ in  \eqref{T2estimate}, we could take a $\tilde c>0$ large enough such that $-\calL^{\rm{qnl}}_\del   ( \tilde c\Phi_2(x) )\geq T_\delta^2(x)$, then from the maximum principle we conclude that
\begin{equation*}
\max_{x \in\Om\cup\Om_\del } (e_\delta^2(x) - \tilde c \Phi_2(x))   \leq  \max_{x\in \Om_\del } (e_\delta^2(x) - \tilde c \Phi_2(x))  \leq 0\,.
\end{equation*}
So we have $e_\delta^2(x) \leq \tilde c \Phi_2(x) \leq \tilde c(\delta+\delta^2)$. Using the same arguments to $-e_\delta^2(x) $ we also have
$-e_\delta^2(x) \leq \tilde c(\delta+\delta^2)$. Thus $|e_\delta^2(x)|=O(\delta)$.
\end{proof}

\section{Numerical discretization and numerical examples}\label{sec_num}
In this section, we will develop a finite difference discretization
and consider several benchmark problems to check the accuracy and
stability performance of the numerical scheme. The patch-test
consistency, symmetry and positive definiteness of the finite
difference matrix are validated numerically.

\subsection{Numerical scheme} \label{subsec_scheme}
We use finite difference for spatial discretization. The domain $\Omega=(-1, 1)$ is divided into $2N$ uniform subintervals
with equal length $h=1/N$ and grid points $-1 = x_0 <x_1<\dots<x_{2N}=1$ so the interface grid point is $x_N = 0$. Homogeneous
Dirichlet boundary condition $u = 0$ is assumed on the boundary domain $\Omega_\del=(-\delta-1,-1)\cup \{1\}$.
 We use the scaling invariance of second moments of $\gamma_{\delta}$ and
local diffusion and approximate the quasinonlocal diffusion
operator $\calL^{\rm{qnl}}_\del $ in the three regimes.
The finite difference scheme we uses is not only a convergent scheme for the QNL problem with fixed $\delta$, but also a convergent scheme for the local
differential equation with fixed ratio between $\delta$ and $h$, thus an asymptotically compatible scheme, a notion developed in \cite{Tian2013a, Tian2014a}.

For simplicity of discussion, we always assume that $\delta/h=r$ with
$r$ being an integer in the following. We discuss in order the discretization scheme in the nonlocal region, transitional region and local region respectively. Special
treatment is used in the transitional region for the scheme to be asymptotically compatible.
\begin{itemize}
\item Case I (nonlocal region): for $i\in \{ 0,1,\cdots, N\}$,
\end{itemize}
\begin{equation}\label{fdm_case1}
\begin{split}
\calL^{\rm{qnl}}_\del  u(x_i)
=&2\int_{-\delta}^{\delta}
\left( u(x_i+s)-u(x_i)\right) \gamma_{\delta}(s)ds\\
=&2\int_{0}^{\delta}
\left(\frac{ u(x_i+s)-2u(x_i)+u(x_i-s)}{s^2}\right) s^2\gamma_{\delta}(s)ds\\
\approx&2\sum_{j=1}^{r}\left(\frac{ u(x_{i+j})-2u(x_i)+u(x_{i-j}) }{(jh)^2}\right)
\int_{(j-1)h}^{jh}s^2\gamma_{\delta}(s)ds\,.
\end{split}
\end{equation}

\begin{itemize}
\item Case II (transitional region): for $i\in \{ N+1,N+2, \cdots, N+r\}$
\end{itemize}
\beq \label{fdm_case2_eq1}
\begin{split}
\calL^{\rm{qnl}}_\del u(x_i)
=&2\int_{x_i}^\delta
 \gamma_{\delta}(\abs{s})\left(u(x_i-s)-u(x_i)\right)ds+2\left( \int_{x_i}^{\delta}  s\gamma_{\delta}(s)ds\right) u^\prime(x_i) \\
 +& \left(2 \int_0^{x_i} s^2\gamma_\delta(|s|)ds  +2x_i \int_{x_i}^\delta s\gamma_\delta(|s|)ds \right) u^{\prime\prime}(x_i).
\end{split}
\eeq
\begin{itemize}
\item[]
Now we split the nonlocal integral term into diffusion part and convection part:
\begin{equation*}
\begin{split}
2\int_{x_i}^\delta &
 \gamma_{\delta}(\abs{s})\left(u(x_i-s)-u(x_i)\right)ds  \\
 &\quad = \int_{x_i}^\delta \gamma_{\delta}(\abs{s})\left(u(x_i+s)-2u(x_i)+u(x_i-s)\right)ds\\
 &\qquad \quad -\int_{x_i}^\delta \gamma_{\delta} \left( u(x_i+s)-u(x_i-s) \right) ds\,.
 \end{split}
\end{equation*}
From here we derive the discretization for $\calL^{\rm{qnl}}_\del u(x_i)$:
\end{itemize}

\beq \label{fdm_case2_eq2}
\begin{split}
\calL^{\rm{qnl}}_\del& u(x_i)
\approx \sum_{j=x_i/h}^{r} \frac{ u(x_{i+j})-2u(x_i)+u(x_{i-j}) }{(jh)^2}\int_{(j-1)h}^{jh}s^2\gamma_{\delta}(s)ds \\
- &\sum_{j=x_i/h}^{r} \frac{ u(x_{i+j})-u(x_{i-j}) }{jh}\int_{(j-1)h}^{jh}s\gamma_{\delta}(s)ds \\
&\qquad \quad \qquad + 2\left( \int_{x_i}^{\delta}  s\gamma_{\delta}(s)ds\right) \frac{u(x_{i+1})-u(x_{i})}{h} \\
+& \left(2 \int_0^{x_i} s^2\gamma_\delta(|s|)ds  +2x_i \int_{x_i}^\delta s\gamma_\delta(|s|)ds \right)  \frac{u(x_{i+1})-2u(x_i)+u(x_{i-1})}{h^2}.
\end{split}
\eeq
\begin{itemize}
\item Case III (local region): for $i\in \{N+r+1,\cdots, 2N\}$,
\end{itemize}
\begin{equation}\label{fdm_case3}
\begin{split}
\calL^{\rm{qnl}}_\del  u(x_i)
=& u^{\prime\prime}(x_i)
\approx  \frac{u(x_{i+1})-2u(x_i)+u(x_{i-1})}{h^2}\,.
\end{split}
\end{equation}

 \begin{remark}
 The finite difference discretization described above is a first order scheme with respect to $h$ for fixed horizon $\del$ to the QNL equation \eqref{eq:QNLD_steady}, as well as
 a first order scheme for fixed ratio $r$ between $\del$ and $h$ to the local equation \eqref{eq:LD_steady}.
We split the convection and diffusion parts in \eqref{fdm_case2_eq1} to balance the convection from
nonlocal and local contributions. The resulting discretized expression \eqref{fdm_case2_eq2} will be asymptotically compatible to the local equation.  Otherwise, direct discretization of \eqref{fdm_case2_eq1} will lead to artificial convention terms and
thus it will cause numerical inconsistency and instability on the interfacial regions. {{To demonstrate this, we compute the nonlocal-local coupling model \eqref{eq:QNLD_steady} with external force \eqref{eq:QNLD_steady_force} and interface $x^*=1/2$ by the direct discretization and the compatible scheme \eqref{fdm_case2_eq2}, respectively. The results are plotted in Figure~\ref{Fig:DvsC}. Notice that the exact gradient is not zero at interface $x^*=1/2$, and thus the artificial convection due to direct discretization causes divergence in the computations.

\begin{figure}[htbp]
\begin{center}
\subfigure[Solution error vs $h$]{\includegraphics[width= 0.3\textwidth]{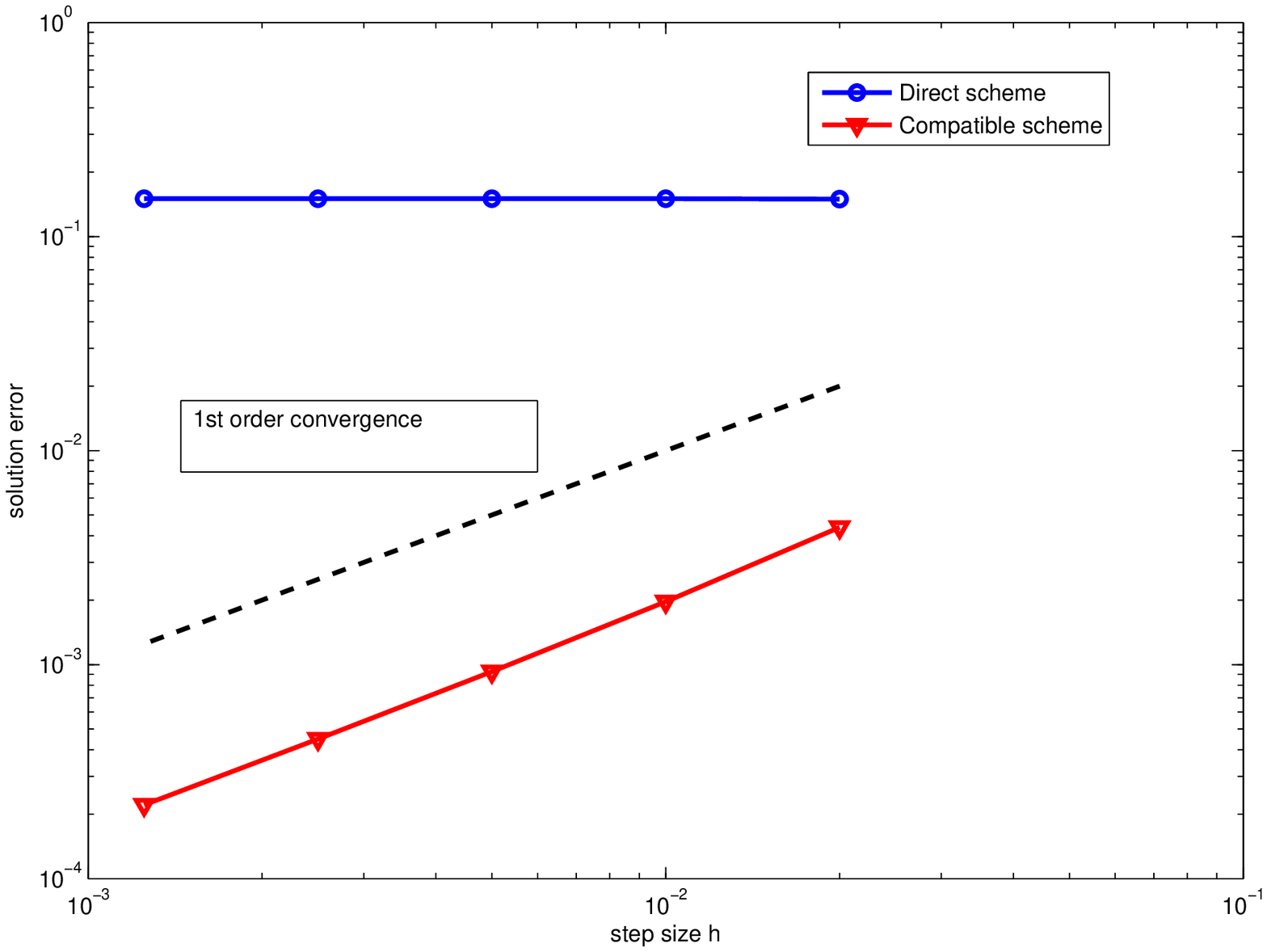}}
\quad
\subfigure[Gradient error vs $h$]{\includegraphics[width=0.3\textwidth ]{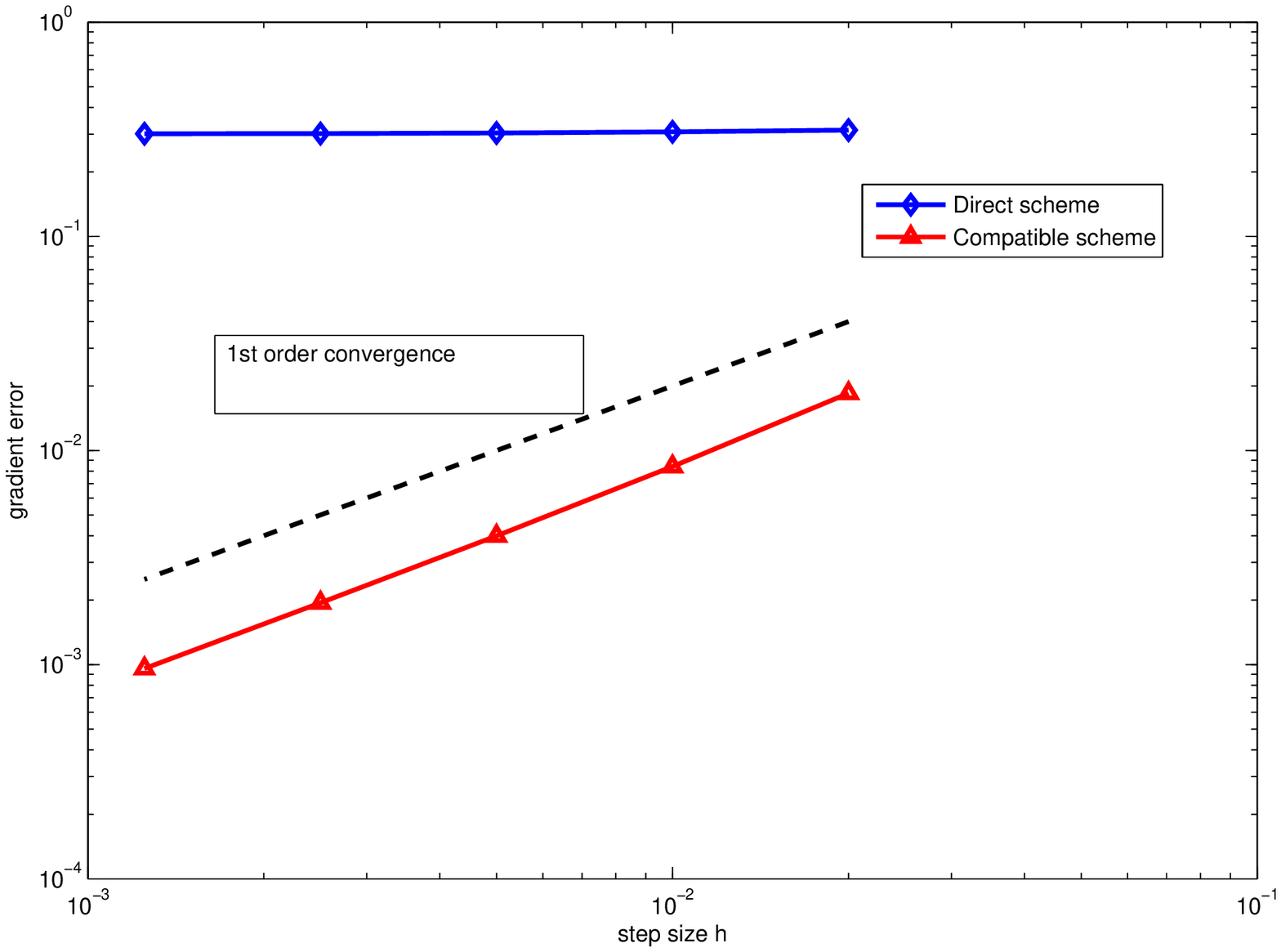}}
\quad
\subfigure[Gradient plots]{\includegraphics[width=0.3\textwidth]{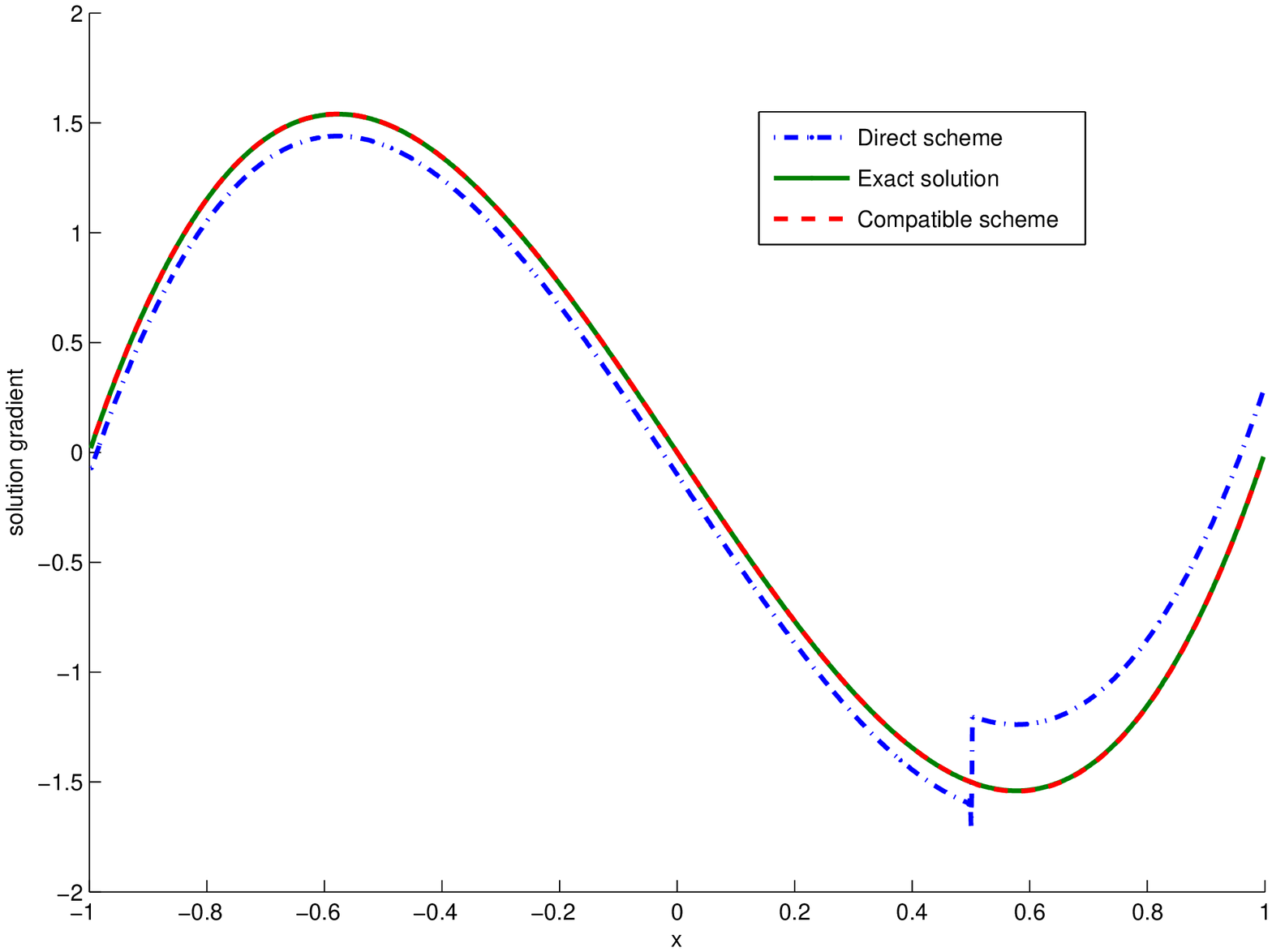}}
\vspace{-0.14 in}
\caption{\small{Plots of the nonlocal-local coupling model \eqref{eq:QNLD_steady} with external force \eqref{eq:QNLD_steady_force}
and interface $x^*=1/2$ computed by the direct method and the compatible scheme \eqref{fdm_case2_eq2}, respectively.
Kernel function is chosen to be $\gamma_{\delta}(x)=\frac{1}{\delta^2|x|}\chi_{(-\del,\del)}(x)$. The ratio between $\del$ and $h$ is fixed to be $\del=3h$.
The exact gradient is not zero at interface $x^*=1/2$, hence the artificial convection due to direct discretization leads to divergence in the computations.
 }}\label{Fig:DvsC}
\end{center}
\end{figure}
}}

\end{remark}

 \begin{remark}
 For the general case that the interface is at $x^\ast \neq 0$, we have the following formulas in replace of equation \eqref{fdm_case2_eq1}.
 If $x^\ast$ is at the left side of the transitional region, then \eqref{fdm_case2_eq1} is replace by
 \begin{equation*}
 \begin{split}
\calL^{\rm{qnl}}_\del u(x_i)
=&2\int_{x_i-x^\ast}^\delta
 \gamma_{\delta}(\abs{s})\left(u(x_i-s)-u(x_i)\right)ds+2\left( \int_{x_i-x^\ast}^{\delta}  s\gamma_{\delta}(s)ds\right) u^\prime(x_i) \\
 +& \left(2 \int_0^{x_i-x^\ast} s^2\gamma_\delta(|s|)ds  +2(x_i-x^\ast) \int_{x_i-x^\ast}^\delta s\gamma_\delta(|s|)ds \right) u^{\prime\prime}(x_i)\,.
\end{split}
 \end{equation*}
 If $x^\ast$ is at the right side of the transitional region, then \eqref{fdm_case2_eq1} is replace by
  \begin{equation*}
 \begin{split}
\calL^{\rm{qnl}}_\del u(x_i)
=&2\int_{x^\ast-x_i}^\delta
 \gamma_{\delta}(\abs{s})\left(u(x_i+s)-u(x_i)\right)ds-2\left( \int_{x^\ast-x_i}^{\delta}  s\gamma_{\delta}(s)ds\right) u^\prime(x_i) \\
 +& \left(2 \int_0^{x^\ast-x_i} s^2\gamma_\delta(|s|)ds  +2(x^\ast-x_i) \int_{x^\ast-x_i}^\delta s\gamma_\delta(|s|)ds \right) u^{\prime\prime}(x_i)\,.
\end{split}
 \end{equation*}

\end{remark}

\subsection{Numerical experiments}
We solve the QNL problem \eqref{eq:QNLD_steady} with right hand side $f$ to be
\begin{equation}\label{eq:QNLD_steady_force}
f(x)=-12 x^2+4\,.
\end{equation}
The exact solution for the limiting local diffusion problem \eqref{eq:LD_steady} is
\begin{equation}\label{exact_local_sol}
u_{0}=(1-x)^2(1+x)^2\quad \text{for }x\in\Omega.
\end{equation}

We adopt the discretization scheme described in section \ref{subsec_scheme} and
compute the QNL solution with the ratio between $\delta$ and spatial step size $h$ to be fixed.
Two types of kernels are used with one being $\gamma_{\delta}(x)=\frac{3}{2\delta^3}\chi_{(-\del,\del)}(x)$ and another being $\gamma_{\delta}(x)=\frac{1}{\delta^2 |x|}\chi_{(-\del,\del)}(x)$. We compute first the $L^{\infty}$ difference between the QNL solution and the local solution and then the $L^{\infty}$ difference of the gradients which are approximated by second order central finite difference at the mesh points.
First order convergences with respect to $h$ are observed in both cases.
The results are listed in Table~\ref{table1} and \ref{table2}. {More careful studies on errors in other norms
and more effective gradient recovery techniques, like those proposed in \cite{DuTao2016}
for nonlocal problems,  will be studied in the future.}

\begin{table}[htp!]
\begin{center}
\begin{tabular}{|c|c|c|c|c|}
\hline
{$h$} & {$\|u_\del^{\rm{qnl}}-u_0\|_{L^{\infty}}$} &Order
&{  $\|(u_\del^{\rm{qnl}}-u_0)^\prime\|_{L^{\infty}}$} & Order \\
\hline
$1/50$ & $  1.56e$-$2$ & $-$ & $1.91e$-$2$&  $-$\\
\hline
$1/100$ & $ 8.07e$-$3$ & $0.95$ & $9.61e$-$3$ & $0.99$\\
\hline
$1/200$ & $ 4.10e$-$3$ & $0.98$ & $4.82e$-$3$& $0.99$\\
\hline
$1/400$ & $  2.06e$-$3$ & $0.99$ & $2.42e$-$3$ & $1.00$\\
\hline
$1/800$ & $  1.04e$-$3$ & $0.99$ & $1.21e$-$3$&   $1.00$ \\
\hline

\end{tabular}
\smallskip
\caption{ $L^{\infty}$ differences of solutions $u_\del^{\rm{qnl}}$ to $u_0$ and their gradients.
We fix $\delta=3h$ and the kernel function is $\gamma_{\delta}(x)=\frac{3}{2\delta^3}\chi_{(-\del,\del)}(x)$.  }\label{table1}
\end{center}
\end{table}

\begin{table}[htp!]
\begin{center}
\begin{tabular}{|c|c|c|c|c|}
\hline
{$h$} & {$\|u_\del^{\rm{qnl}}-u_0\|_{L^{\infty}}$} &Order
&{  $\|(u_\del^{\rm{qnl}}-u_0)^\prime\|_{L^{\infty}}$} & Order \\
\hline
$1/50$ & $  1.19e$-$2$ & $-$ & $1.80e$-$2$&  $-$\\
\hline
$1/100$ & $ 6.19e$-$3$ & $0.95$ & $9.13e$-$3$ & $0.98$\\
\hline
$1/200$ & $ 3.14e$-$3$ & $0.97$ & $4.59e$-$3$& $0.99$\\
\hline
$1/400$ & $  1.59e$-$3$ & $0.99$ & $2.30e$-$3$ & $1.00$\\
\hline
$1/800$ & $  7.97e$-$4$ & $0.99$ & $1.15e$-$3$&   $1.00$ \\
\hline

\end{tabular}
\smallskip
\caption{$L^{\infty}$ differences of solutions $u_\del^{\rm{qnl}}$ to $u_0$ and their gradients.
We fix $\delta=3h$ and the kernel function is $\gamma_{\delta}(x)=\frac{1}{\delta^2|x|}\chi_{(-\del,\del)}(x)$.}\label{table2}
\end{center}
\end{table}

\subsection{Local-nonlocal-local coupling}
Volumetric constraints for nonlocal models often cause non-physical
boundary layer issues, as shown in Figure \ref{Fig:LtoNtoL_demo}
$(a)$.  We could fix the boundary layer problem by coupling the
nonlocal models with local models and remove the volume constraints
completely.  Figure \ref{Fig:LtoNtoL_demo} $(b)$ shows the solution of
the local-nonlocal-local coupling with interfaces at
$x^a=\frac{-1}{2}$ and $x^b=\frac{1}{2}$.  We see that the coupling
method removes the artificial boundary layer caused by volume
constraints with classical local Dirichlet boundary conditions
imposed.
\begin{figure}[htbp]
\begin{center}
\subfigure[Nonlocal-local coupling model]{\includegraphics[height=  4.5 cm, width= 5.7 cm]{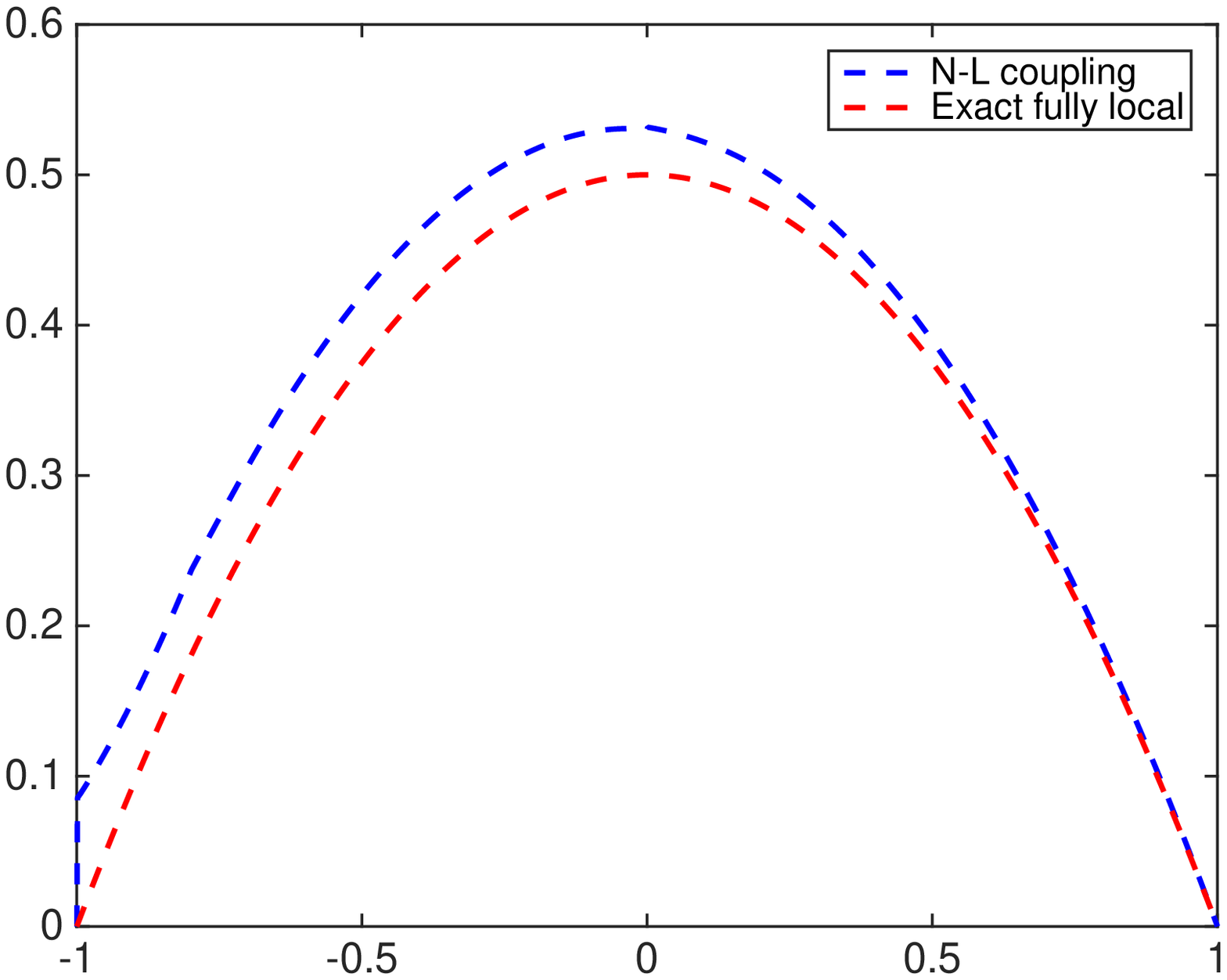}}
\qquad
\subfigure[Local-nonlocal-local coupling model]{\includegraphics[height=4.5  cm, width = 5.7 cm]{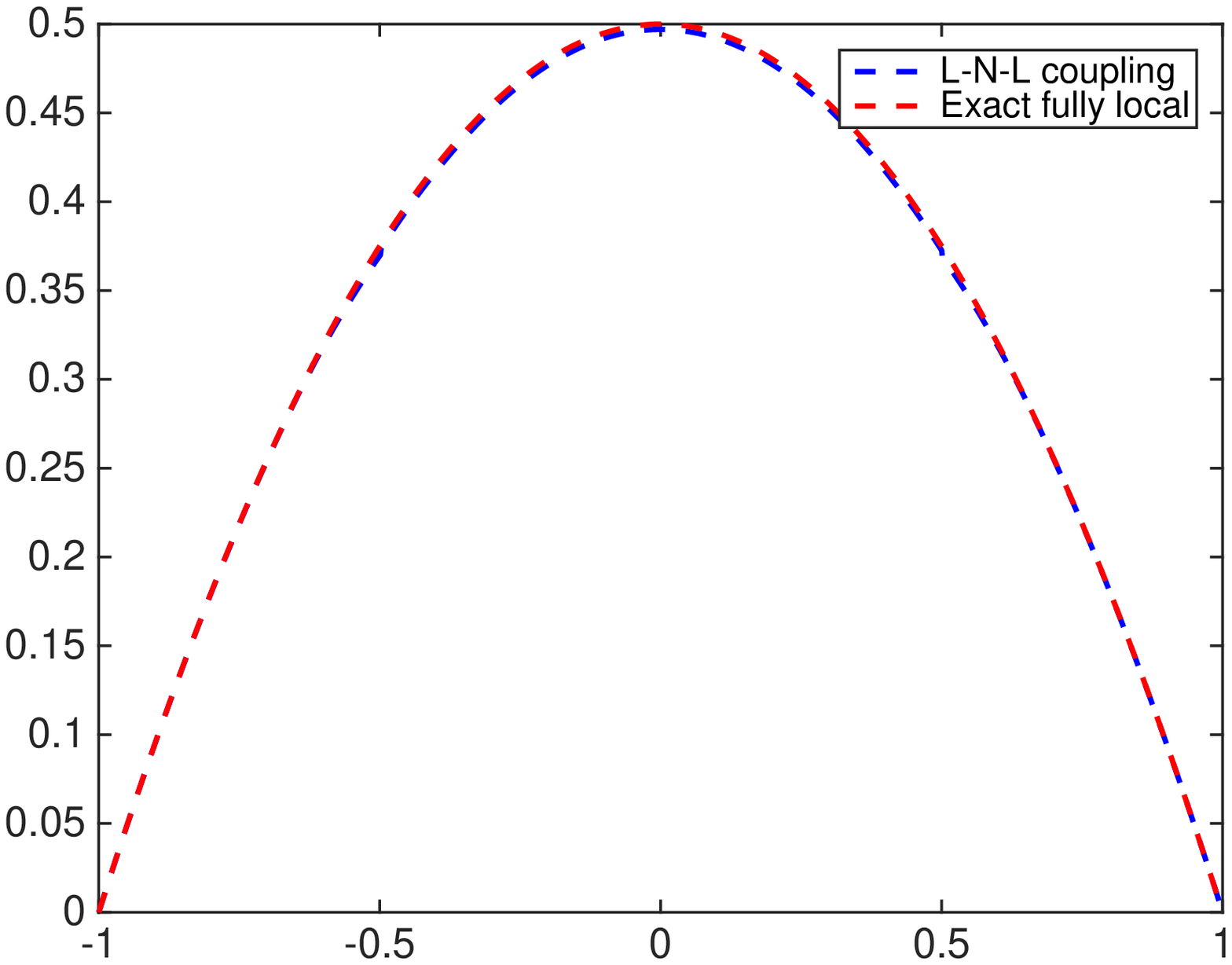}}
\vspace{-0.14 in}
\caption{\small{Plots of solutions to nonlocal-local coupling model, local-nonlocal-local coupling model and fully local local model with homogeneous Dirichlet boundary condition and right hand side $f\equiv 1$.  Kernel function is chosen to be $\gamma_{\delta}(x)=\frac{3}{2\delta^3}\chi_{(-\del,\del)}(x)$. The nonlocal-local coupling model has interface at $x=0$. The
local-nonlocal-local coupling model has interfaces at $x^a=\frac{-1}{2}$ and $x^b=\frac{1}{2}$. The mesh size is $h=1/800$,
the horizon size of nonlocal interaction is $\delta=0.2$.
The nonlocal-local coupling model displays non-physical boundary layer at the nonlocal side, whereas the result of
local-nonlocal-local removes the boundary layer.
 }}\label{Fig:LtoNtoL_demo}
\end{center}
\end{figure}

Next we consider the following singular external forces:
\begin{equation}\label{singular_force}
f(x)=\frac{(1-x^2)(1+x^2)}{|x-x^*|},\quad x^*=h/2.
\end{equation}
The solutions for fully nonlocal, local-nonlocal-local coupling and classical local models are plotted in Figure~\ref{Figure:defect_external_force}.
We can see that the local-nonlocal-local coupling not only captures the singular behavior of the nonlocal solution at $x^\ast$, but also
matches with the local solution at two sides of the bar $(-1,1)$.

\begin{figure}[htp!]
\centering
\includegraphics[height =5.5 cm]{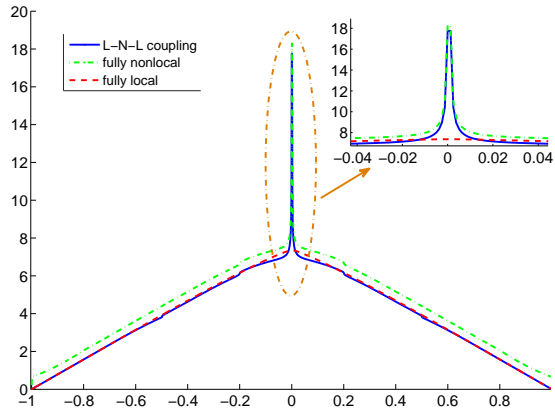}
\caption{The solutions are plotted with mesh size $h=1/800$,
and horizon size $\delta=0.2$. We fix the local-nonlocal-local coupling with interfaces at $x^a=\frac{-1}{2}$ and $x^b=\frac{1}{2}$.  Kernel function is chosen to be $\gamma_{\delta}(x)=\frac{3}{2\delta^3}\chi_{(-\del,\del)}(x)$ }\label{Figure:defect_external_force}
\end{figure}

\section{Conclusion}\label{sec_conclusion}
By extending the idea of ``geometric reconstruction'' proposed in
\cite{E:2006,LiLu2016}, we developed a top-down quasinonlocal coupling
method to study the nonlocal-to-local (NtL) diffusion problem in one
dimensional space. This new coupling framework removes interfacial
inconsistency and maintains all physical properties at local continuum
PDE levels, whereas none of existing coupling methods for
nonlocal-to-local problems satisfies all of these properties.
We proved the well-posedness of the coupling
problem by a quasinonlocal version of the Poincar\'e inequality and
established rigorous estimate of the modeling error by the maximum
principle. Furthermore, we proposed a first order finite difference
numerical discretization and confirmed the analysis by several
numerical tests.  The coupling formulation also removes  artificial
boundary effects caused by the fully nonlocal model when only
classical Dirichlet boundary conditions are imposed.
Although our discussions here have focused on the scalar one dimensional model
problems, it is natural to investigate whether similar ideas can be
developed for systems of equations and for problems defined in
multi-dimensions. {Such generalization is indeed possible, partly
because of the fact that the nonlocal diffusion models considered here
are based on pairwise interactions, the cases
that have been explored in the atomistic-to-continuum coupling methods, see for example \cite{Shapeev2012a,Shapeev2012b}.
Further investigations will be carried out in our follow-up
 works along this direction and for nonlocal problems possibly involving more general interactions.  }



\bibliographystyle{abbrv}
\bibliography{QNLcouple}

\end{document}